\documentclass{amsart}
\usepackage{amssymb, mathrsfs, amsthm, amsmath}

\title[Lower bounds for canonical heights on Drinfeld modules]{A lower bound for the canonical height associated to a Drinfeld module}\author{Patrick Ingram}
\email{pingram@math.colostate.edu}
\address{Department of Mathematics, Colorado State University, Fort Collins, USA}

\newcommand{\QQ}{\mathbb{Q}}
\newcommand{\ZZ}{\mathbb{Z}}
\newcommand{\CC}{\mathbb{C}}
\newcommand{\RR}{\mathbb{R}}
\newcommand{\FF}{\mathbb{F}}
\newcommand{\PP}{\mathbb{P}}
\renewcommand{\AA}{\mathbb{A}}
\newcommand{\Ocal}{\mathcal{O}}
\newcommand{\Gal}{\operatorname{Gal}}
\newcommand{\Spec}{\operatorname{Spec}}
\newcommand{\Div}{\operatorname{Div}}
\newcommand{\lcm}{\operatorname{lcm}}

\newcommand{\Pic}{\operatorname{Pic}}

\newcommand{\Disc}{\mathscr{D}}

\newcommand{\Aberk}[1]{\mathbb{A}^1_{\mathrm{Berk}}}
\newcommand{\Pberk}[1]{\mathbb{P}^1_{\mathrm{Berk}}}

\newcommand{\reas}[1]{#1^c}

\renewcommand{\epsilon}{\varepsilon}

\newtheorem{theorem}{Theorem}[section]
\newtheorem{proposition}[theorem]{Proposition}
\newtheorem{lemma}[theorem]{Lemma}
\newtheorem{conjecture}[theorem]{Conjecture}
\newtheorem{corollary}[theorem]{Corollary}

\theoremstyle{remark}
\newtheorem{remark}[theorem]{Remark}

\theoremstyle{definition}
\newtheorem{definition}[theorem]{Definition}

\begin{document}
\begin{abstract}
Denis associated to each Drinfeld module $\phi$ over a global function function field $L$ a canonical height function $\hat{h}_\phi$, which plays a role analogous to that of the N\'{e}ron-Tate height in the context of elliptic curves.  We prove that there exists a constant $\epsilon>0$, depending only on the number of places at which $\phi$ has bad reduction, such that either $x\in \phi(L)$ is a torsion point of bounded order, or else
\[\hat{h}_\phi(x)\geq \epsilon \max\{h(j_\phi), \deg(\Disc_{\phi/L})\},\]
where $j_\phi$ and $\Disc_{\phi/L}$ are analogues of the $j$-invariant and minimal discriminant of an elliptic curve.  As an application, we make some observations about specializations of one-parameter families of Drinfeld modules.
\end{abstract}
\maketitle

\section{Introduction}
\label{sec:intro}

Let $X/\FF_q$ be a curve, and let $A\subseteq K=\FF_q(X)$ be the ring of functions which are regular at some specified point $\infty\in X(\FF_q)$.  If $L/K$ is a finite extension, then a \emph{Drinfeld $A$-module} $\phi/L$ is a ring homomorphism $\phi:A\to\operatorname{End}_L(\mathbb{G}_\mathrm{a})$ satisfying certain additional conditions, and we denote the associated $A$-module structure on $\mathbb{G}_\mathrm{a}(L)$ by $\phi(L)$.  Denis \cite{denis} associated to such an object a canonical height function $\hat{h}_\phi:\phi(L)\to\RR$, differing from the usual Weil height by at most a constant, and satisfying $\hat{h}_\phi(\phi_a(x))=|a|_\infty^r\hat{h}_\phi(x)$ for each $x\in\phi(L)$ and $a\in A$, where $|\cdot|_\infty$ is the $\infty$-adic absolute value, and $r\geq 1$ is the \emph{rank} of $\phi$.  As a consequence of these two properties, the height $\hat{h}_\phi$ vanishes precisely on the torsion submodule of $\phi(L)$.

Our main result gives a lower bound on the canonical height associated to $\phi$ in terms of two quantities defined below, namely the $j$-invariant $j_\phi$ of $\phi$, and the minimal disciminant $\Disc_{\phi/L}$, as well as the number of places at which $\phi$ has bad reduction.  The quantities $j_\phi$ and $\Disc_{\phi/L}$ are defined by analogy to the case of elliptic curves, and Theorem~\ref{th:jplaces} can be seen as an analogue of a result of Silverman \cite{joebound} in that context (see also \cite{mebound} in the context of the dynamics of unicritical polynomial maps).  Indeed, just as in the results of \cite{joebound}, we need only count places at which $\phi$ has bad reduction that is not potentially good.
\begin{theorem}\label{th:jplaces}
For every $s, r\geq 1$ there exist a constant $\epsilon>0$, and an ideal $\mathfrak{a}\subseteq A$, such that for every Drinfeld module $\phi/L$ of rank $r$ with at most $s$ places of persistently bad reduction, and every $x\in\phi(L)$, either $x\in\phi[\mathfrak{a}]$ or else
\[\hat{h}_\phi(x)\geq \epsilon \max\{h(j_\phi), \deg(\Disc_{\phi/L})\}.\]
\end{theorem}
We immediately recover a theorem of Ghioca \cite{ghioca1} bounding the torsion submodule of a Drinfeld module in terms of the number of places of bad reduction.
\begin{corollary}[Ghioca~\cite{ghioca1}]\label{th:numberofplaces}
Let $\phi$ be a Drinfeld module over the global function field $L$.  Then $\#\phi^{\mathrm{Tors}}(L)$ is bounded by a quantity depending only on $L$, $q$, the rank of $\phi$, and the number of places at which $\phi$ has bad reduction.
\end{corollary}

We note that the result in \cite{ghioca1} also involves a lower bound on the canonical height, and the methods introduced there contribute to the proof of Theorem~\ref{th:jplaces}.  The lower bound produced in \cite{ghioca1} is, however, of the form $\hat{h}_\phi(x)\geq \epsilon$, where $\epsilon$ depends on the number of places of bad reduction, and hence is fundamentally weaker than that presented in Theorem~\ref{th:jplaces}.  
For example, we will show below that for any $B\geq 1$, there are only finitely many $L$-isomorphism classes of Drinfeld modules $\phi/L$ of a given rank satisfying
\[ \max\{h(j_\phi), \deg(\Disc_{\phi/L})\}\leq B.\]
The lower bound in Theorem~\ref{th:jplaces}, therefore, becomes arbitrarily large as one varies the Drinfeld module $\phi/L$.  The analogue of Theorem~\ref{th:jplaces} for elliptic curves is a key ingredient in Silverman's quantitative version of Siegel's Theorem \cite{jhs-siegel}, and we suspect that Theorem~\ref{th:jplaces} will have similar applications.  Motivated by this, we propose the following conjecture, analogous to conjectures of Lang \cite[Conjecture~VIII.9.9]{aec} and Silverman~\cite[Conjecture~4.98]{ads}.  
\begin{conjecture}\label{conj:lang}
For every finite extension $L/K$ and every $r\geq 1$ there is an $\epsilon>0$ such that if $\phi/L$ is a Drinfeld module of rank $r$, and $x\in\phi(L)$ is non-torsion, then
\[\hat{h}_\phi(x)\geq \epsilon\max\{h(j_\phi), \deg(\Disc_{\phi/L})\}.\]
\end{conjecture}
Another way to state of Theorem~\ref{th:jplaces}, then, is that Conjecture~\ref{conj:lang} holds if one restricts attention to Drinfeld modules with potentially good reduction at all but a bounded number of places.  We note that, if true, Conjecture~\ref{conj:lang} is essentially the best-possible form of lower bound, at least if one assumes the weaker of Poonen's Uniform Boundedness Conjectures for Drinfeld modules \cite[Conjecture~1]{poonenunif}.
\begin{theorem}\label{th:sharpness}
Let $L/K$ be a finite extension, let $r\geq 1$, and suppose that there is a uniform bound on the size of the torsion submodules $\phi^{\mathrm{Tors}}(L)$ as $\phi/L$ varies over Drinfeld modules of rank $r$.  Then for every Drinfeld module $\phi/L$ of rank $r$, there is a non-torsion $x\in \phi(L)$ such that
\[\hat{h}_\phi(x)\leq 2\max\{h(j_\phi), \deg(\Disc_{\phi/L})\}+O(1),\]
where the implied constant depends only on $L$ and $r$.
\end{theorem}
Note that, by Corollary~\ref{th:numberofplaces}, Theorem~\ref{th:sharpness} can be made unconditional if one restricts attention to Drinfeld modules with potentially good reduction at all but a bounded number of places, showing that Theorem~\ref{th:jplaces} is essentially the best possible result in this context.

As an application of  Theorem~\ref{th:jplaces}, we consider the problem of specializing a family of Drinfeld modules.  If $C/L$ is a curve, and if  $\phi/L(C)$ is a Drinfeld module,  one might ask how the algebraic structure of the fibre $\phi_\beta(L)$ varies with $\beta\in C(L)$, omitting values at which the coefficients of $\phi$ are not defined.  For example, if $x\in \phi(L(C))$ is non-torsion, then for each $\beta\in L(C)$ we obtain a specialization homomorphism $\sigma_\beta: \langle x\rangle\to \Gamma_{\phi_\beta}(x_\beta, L)$, where 
\[\Gamma_{\phi_\beta}(x_\beta, L)=\{z\in \phi_\beta(L):\phi_{\beta, a}(z)=\phi_{\beta, b}(x_\beta)\text{ for some }a, b\in A\}\]
is the largest submodule of $\phi_\beta(L)$ onto which one could hope that $\sigma_\beta$ would surject, that is, the largest  submodule of $\phi(L)$ containing $x_\beta$ but having the same rank as $\langle x_\beta\rangle$.  One could ask how often $\sigma_\beta$ is injective, and how close $\sigma_\beta$ is to being surjective.
\begin{theorem}\label{th:fams}
Let $\phi/L(C)$ be a non-isotrivial Drinfeld $A$-module, and let $x\in \phi(L(C))$ be non-torsion.
\begin{enumerate}
\item The specialization homomorphism is injective outside of a set of bounded height.
\item There exists an $E\in\Div(C)$ and a finite set $S\subseteq M_L$ of places such that the index $(\Gamma_{\phi_\beta}(x_\beta, L):\mathrm{im}(\sigma_\beta))$ is bounded by a quantity which depends only on $\phi$, $x$, and the number of places $v\not\in S$ at which $\beta$ is not integral relative to $E$.
\end{enumerate}
\end{theorem}
Theorem~\ref{th:fams} is similar in flavour to two results of Silverman~\cite{jhs1, jhs2} in the context of abelian varieties.

In the course of our study of local heights, we prove a result which may be of independent interest.  Denis~\cite{denis}, in the paper defining the canonical height $\hat{h}_\phi$ associated to $\phi$, shows that the difference between the canonical height and the usual height is bounded, and Theorem~\ref{th:zimmer} makes this bound explicit.  This is analogous to result for elliptic curves that was made explicit independently by Dem'janenko~\cite{demjanenko} and Zimmer~\cite{zimmer}, and may be useful for computations involving Drinfeld modules.
\begin{theorem}\label{th:zimmer}
Let $L$ be a global function field and let $\phi/L$ be a Drinfeld module of rank $r\geq 1$.  There exist explicit constants $C_1$ and $C_2$, depending only on $A$, $L$, and $r$, such that
Then for all $x\in L^{\mathrm{sep}}$, we have
\[-C_1-\left(\frac{1}{1-q^{-r}}\right)h(\phi)\leq \hat{h}_\phi(x)-h(x)\leq h(\phi)+C_2,\]
where $h(\phi)$ is defined in Section~\ref{sec:global}.  If $L$ is a finite extension of the field $K$ of fractions of $A$, and the heights above are relative to $K$, then $C_1$ and $C_2$ depend only on $A$ and $r$.
\end{theorem}
With an eye to Theorem~\ref{th:jplaces} and Conjecture~\ref{conj:lang}, we note that while $h(\phi)$ is not the same as $\max\{h(j_\phi), \deg(\Disc_{\phi/L})\}$, every Drinfeld module is $L$-isomorphic to a model for which these two quantities are commensurate.

In Section~\ref{sec:prelims} we review some of the basic definitions of Drinfeld modules and heights, and introduce the $j$-invariant of a Drinfeld module.  In Section~\ref{sec:localfields} we look more closely at Drinfeld modules over local fields.  Section~\ref{sec:localheights} is devoted to the study of a local height function on a Drinfeld module over a local field, which is related to, but not the same as, the local heights introduced by Poonen \cite{poonen}.  Our local height functions have the advantage of being coordinate invariant, and more closely resembling the local heights on elliptic curves.  In Section~\ref{sec:global} we return to Drinfeld modules over global fields, defining the minimal discriminant $\Disc_{\phi/L}$, and proving Theorem~\ref{th:jplaces}, Theorem~\ref{th:sharpness}, and Theorem~\ref{th:zimmer}.  Finally, in Section~\ref{sec:families}, we prove Theorem~\ref{th:fams}.


\section{Notation and preliminaries}
\label{sec:prelims}

\subsection{Drinfeld modules}
Throughout, we suppose that $q$ is a power of a prime, and that $K$ is a function field in one variable over $\FF_q$, that is, that $K=\FF_q(X)$ for some algebraic curve $X/\FF_q$.  We fix a place $\infty\in X(\FF_q)$, and let $A\subseteq K$ denote the ring of regular functions at $\infty$.  By an \emph{$A$-field}, we mean a field $L$ with a homomorphism $i:A\to L$, and we will consider only the case in which $L$ has \emph{generic characteristic}, that is, where $i$ is an injection; the typical example is where $L/K$ is a finite extension, and $i$ is the inclusion map.

If $L$ is an $A$-field, then a \emph{Drinfeld} $A$-module over $L$ is a homomorphism $\phi:A\to \mathrm{End}_L(\mathbb{G}_\mathrm{a})$, $a\mapsto \phi_a$, with the property that, for all $a\in A$,
\[\phi_a(x)=ax+O(x^q)\in L[x],\]
but $\phi_a(x)\neq ax$ for at least some $a\in A$.  For $a\in A$, we let $|a|_\infty=\#(A/a)$, and recall that Drinfeld \cite{drinfeld} proved the existence of an integer $r\geq 1$ such that $\deg(\phi_a(x))=|a|_\infty^r$, for all $a\in A$.  This quantity will be known as the \emph{rank} of $\phi$.  For convenience we will define $\deg:A\to \ZZ$ by $|a|_\infty=q^{\deg(a)}$, noting that this agrees with the usual definition in the case $A=\FF_q[T]$.  If $\mathfrak{a}\subseteq A$ is any ideal, then we will write $\phi[\mathfrak{a}]$ for the submodule of $\phi(\overline{L})$ consisting of elements annihilated by $\mathfrak{a}$.

Two Drinfeld modules $\phi/L$ and $\psi/L$ are said to be \emph{isomorphic} over the extension $E/L$ if and only if there exists an $\alpha\in E$ such that $\phi_a(\alpha x)=\alpha\psi_a(x)$ for all $a\in A$, abbreviated $\phi\alpha=\alpha\psi$.  Suppose that we have fixed a ordered set of generators $T_1, ..., T_m$ for $A$ as an $\FF_q$-algebra.  Then we have, for each $i$,
\[\phi_{T_i}(x)=T_i+a_{i, 1}x^q+\cdots +a_{i, \deg(T_i)r}x^{q^{\deg(T_i)r}},\]
with $a_{i, j}\in L$, and $\phi$ is entirely determined by these values.  By the \emph{$\vec{w}$-weighted projective space} $\PP^{\vec{w}}$, where $\vec{w}=(w_1, ..., w_{m+1})\in (\ZZ^+)^{m+1}$, we mean the quotient of $\AA^{m+1}\setminus\{(0, 0, ..., 0)\}$ under the $\mathbb{G}_\mathrm{m}(\overline{L})$ action
\[(x_1, x_2, ..., x_{m+1})\to(\alpha^{w_1}x_1, \alpha^{w_2}x_2, ..., \alpha^{w_{m+1}}x_{m+1}).\]
We warn the reader that, in general, points of $\PP^{\vec{w}}$ which are fixed by $\Gal(\overline{L}/L)$ do not necessarily have a representative with coordinates in $L$, unlike in the case of the usual projective space.  The following definition is also made in \cite{juliapaper}.
\begin{definition}
Fix a set of generators $T_1, ..., T_m$ for $A$ as an $\FF_q$-algebra, fix $r\geq 1$, and let
 $M_{A, r}$ denote the weighted projective space with coordinates $x_{i, j}$, for $1\leq i\leq m$ and $1\leq j\leq r\deg(T_i)$, such that $x_{i, j}$ is given weight $q^j-1$.  If $\phi/L$ is a Drinfeld $A$-module of rank $r\geq 1$, then by the \emph{$j$-invariant} of $\phi/L$, we mean the point
\[j_\phi=[a_{1, 1}, a_{1, 2}, ..., a_{1, \deg(T_1)r}, a_{2, 1}, ..., a_{2, \deg(T_2)r}, ..., a_{m, 1}, a_{m, 2}, ..., a_{m, \deg(T_m)r}]\]
 in $M_{A, r}(L)$.
 \end{definition}
If $A=\FF_q[T]$, and $\phi/L$ is the Drinfeld module of rank 2 defined by
\[\phi_T(x)=Tx+gx^q+\Delta x^{q^2},\]
then it is conventional to set $j=g^{q+1}/\Delta$.  We note that this is the image of the point $j_\phi$, as defined here, under the map $\PP^{(q-1, q^2-1)}\to \PP^1$ given by $[x:y]\mapsto [x^{q+1}:y]$.  So our $j$-invariant nearly exactly generalizes the conventional notion in the rank 2 case.
Note that our invariant is similar, but not identical, to the $j$-invariant defined by Potemine~\cite{potemine}. 
\begin{lemma}
Let $\phi/L$ and $\psi/L$ be two Drinfeld $A$-modules.  Then the following are equivalent:
\begin{enumerate}
\item $\phi$ and $\psi$ are $L^{\mathrm{sep}}$-isomorphic;
\item $\phi$ and $\psi$ are $\overline{L}$-isomorphic;
\item $j_\phi=j_\psi$.
\end{enumerate}
\end{lemma}

\begin{proof}
Since $L^{\mathrm{sep}}\subseteq\overline{L}$, it is clear that condition (1) implies condition (2).  Suppose that (2) is true, and that $\alpha\in\overline{L}$ satisfies $\alpha\psi=\phi\alpha$.  Then if
\[\phi_{T_i}(x)=T_ix+a_{i, 1}x^q+\cdots +a_{i, \deg(T_i)r}x^{q^{\deg(T_i)r}}\]
and
\[\psi_{T_i}(x)=T_ix+b_{i, 1}x^q+\cdots +b_{i, \deg(T_i)r}x^{q^{\deg(T_i)r}},\]
one checks easily that  $b_{i, j}=\alpha^{q^j-1}a_{i, j}$ for all $i$ and $j$.
This is precisely what it means for two points in $\PP^{\vec{w}}$ to be equal.

Finally, suppose that $j_\phi=j_\psi$.  Then if the $a_{i, j}$ and $b_{i, j}$ are defined as above,  there exists an $\alpha\in\overline{L}^*$ such that $a_{i, j}=\alpha^{q^j-1}b_{i, j}$, for all $i$ and $j$.  Choosing some $i$ and  $j$ with $b_{j, i}\neq 0$, we see that $\alpha$ is a root of the polynomial $b_{i, j}X^{q^j-1}-a_{i, j}\in L[X]$, and so in particular $\alpha\in L^{\mathrm{sep}}$.  One now confirms easily that $\alpha\psi_{a}(x)=\phi_{a}(\alpha x)$ for all $a\in \{T_1, ..., T_m\}$, and hence the same is true for all $a\in A$, since $T_1, ..., T_m$ generate $A$ over $\FF_q$.  It follows that $\alpha \psi=\phi\alpha$, and hence that $\phi$ and $\psi$ are $L^{\mathrm{sep}}$-isomorphic.
\end{proof}

\subsection{Heights and valuations}

Throughout, we will make the convention of normalizing logarithms so that $\log q=1$, and write $\log^+ x$ for $\max\{0, \log x\}$.
Let $L$ be a \emph{global function field}, by which we mean a field equipped with a set of non-trivial, non-archimedean absolute values $M_L$ satisfying the product formula:
\[\sum_{v\in M_L}\log|x|_v=0\]
for all $x\in L^*$.  If $L'/L$ is a finite extension, we normalize the elements of $M_{L'}$ so that if $w\in M_L$ is the place below $v\in M_{L'}$, then  $|x|_v=|x|_w$ for all $x\in L$.  A finite extension $L'/L$ will be called \emph{reasonable} if
\[\sum_{\substack{w\in M_{L'}\\w\mid v}}[L'_v:L_v]=[L':L]\]
for every place $v\in M_{L'}$, and we note that every separable extension is reasonable (see \cite[p.~13]{localfields}).  Since composita of reasonable extensions are again reasonable, we will define the \emph{reasonable closure} $\reas{L}$ to be the union of all finite reasonable extensions of $L$.  Note that $L^\mathrm{sep}\subseteq \reas{L}\subseteq \overline{L}$.  The typical case will be that in which $L/K$ is a finite extension, in which case we have $\reas{L}=\overline{L}$.

Fixing a \emph{ground field} $L$, if $L'/L$ is a finite reasonable extension, then for any $x\in L'$, we will define the \emph{height} of $x$ to be
\[h(x)=\sum_{v\in M_{L'}}\frac{[L'_v:L_v]}{[L':L]}\log^+|x|_v,\]
where $L_v$, as usual, denotes the completion of $L$ at $v$.  Note that $h$ extends to a well-defined function $h:\reas{L}\to \RR$.

 More generally, if $\PP^{\vec{w}}$ is the weighted projective space with weights $\vec{w}=(w_0, ..., w_N)$ (all non-zero), then we define a height on $\PP^{\vec{w}}$ by
\[h([x_0:\cdots:x_N])=\sum_{v\in M_{L'}}\frac{[L'_v:L_v]}{[L':L]}\log\max\{|x_0|_v^{1/w_0},\cdots |x_1|_v^{1/w_N}\}.\]
Note that, considering the morphism $\Phi:\PP^{\vec{w}}\to \PP^{N}$ given by
\[\Phi([x_0:x_1:\cdots :x_N])=[x_0^{w_1w_2\cdots w_N}:x_1^{w_0w_2\cdots w_N}:\cdots :x_N^{w_0w_1\cdots w_{N-1}}],\]
we see that sets of bounded height in $\PP^{\vec{w}}(L)$ map to sets of bounded height in $\PP^N(L)$, and hence are finite.
In particular, the standard Northcott finiteness property for $\PP^N$ yields the following useful observation.
\begin{proposition}
For any $B\geq 1$, there are only finite many $\overline{L}$-isomorphism classes of $\phi/L$ with $h(j_\phi)\leq B$.
\end{proposition}


\section{Drinfeld modules over local fields}
\label{sec:localfields}

\subsection{Green's functions}
Let $L$ be a complete local field, equipped with the non-trivial non-archimedean absolute value $|\cdot|$.  To avoid confusion, we will denote by $|\cdot|_\infty$ the absolute value on $A\subseteq K$ corresponding to the place at $\infty$.  For any polynomial
\[f(x)=a_0+a_1x+\cdots +a_dx^d\in L[x],\]
where $d\geq 2$ and $a_d\neq 0$, we define the \emph{Green's function associated to $f$} to be the real-valued function  defined by
\begin{equation}\label{eq:greendef}G_f(x)=\lim_{N\to\infty}d^{-N}\log^+|f^N(x)|.\end{equation}
The following lemma is standard.
\begin{lemma}\label{lem:greensprops}
Let $f(x)\in L[x]$ be a polynomial as above, and let $G_f$ be the Green's function associated to $f$.  Then
\begin{enumerate}
\item $G_f:L\to \RR$ is everywhere defined and non-negative;
\item $G_f(f(x))=\deg(f)G_f(x)$ for all $x\in L$;
\item $G_f(x)=\log|x|+\frac{1}{d-1}\log|a_d|$ whenever
\begin{equation}\label{eq:greenbasin}|x|>\log\max\left\{|a_d|^{-1/(d-1)}, |a_i/a_d|^{1/(d-i)}\right\}.\end{equation}
\end{enumerate}
\end{lemma}

\begin{proof}
The non-negativity of $G_f$, where it is defined, follows simply from the fact that it is a pointwise limit of non-negative functions.
It is easy to check statement (iii), since the condition \eqref{eq:greenbasin} ensures that $|f(x)|=|a_dx^d|>|x|$.  The fact that the limit \eqref{eq:greendef} exists for all $x$ is now straightforward, since if the condition \eqref{eq:greenbasin} occurs for $f^N(x)$, for any $N$, we have an explicit value for the limit, and otherwise it is clear that $G_f(x)=0$.  Part (ii) follows directly from the definition, once we know that the limit always exists.
\end{proof}

Now let $L$ be a complete local $A$-field, and let $\phi/L$ be a Drinfeld $A$-module of rank $r$.
  Then for each non-constant $a\in A$, we have $\deg(\phi_a)=|a|_\infty^r\geq 2$, and so we may associate to the polynomial $\phi_a(x)$ a Green's function $G_{\phi_a}$.
  The following lemma is well-known, although we include a short proof for the convenience of the reader.  It essentially Proposition~3 of \cite{poonen}, in different notation, or the well-known fact in holomorphic dynamics that commuting maps have identical Green's functions.

\begin{lemma}
For any non-constant $a, b\in A$, we have $G_{\phi_a}=G_{\phi_b}$.
\end{lemma}

\begin{proof}
First note that there exists a constant $C$, depending on $b$, such that
\[ \deg(\phi_b)\log^+|x|-C\leq \log^+|\phi_b(x)|_v\leq \deg(\phi_b)\log^+|x|+C.\]
Since $\phi_a$ and $\phi_b$ commute, it follows that
\begin{eqnarray*}
G_{\phi_a}(\phi_b(x))&=&\lim_{N\to\infty}\deg(\phi_a)^{-N}\log^+|\phi_a^N(\phi_b(x))|\\
&=&\lim_{N\to\infty}\deg(\phi_a)^{-N}\log^+|\phi_b(\phi_a^N(x))|\\
&=&\lim_{N\to\infty}\deg(\phi_a)^{-N}\deg(\phi_b)\log^+|\phi_a^N(x)|\\
&=&\deg(\phi_b)G_{\phi_a}(x).
\end{eqnarray*}
But $G_{\phi_a}(x)$ and $\log^+|x|$ also differ only by a bounded amount, by the previous lemma, and so we also have
\begin{eqnarray*}
G_{\phi_b}(x)&=&\lim_{N\to\infty}\deg(\phi_b)^{-N}\log^+|\phi_b^N(x)|\\
&=&\lim_{N\to\infty}\deg(\phi_b)^{-N}G_{\phi_a}(\phi_b^N(x))\\
&=&\lim_{N\to\infty}\deg(\phi_b)^{-N}\deg(\phi_b)^NG_{\phi_a}(x)\\
&=&G_{\phi_a}(x).
\end{eqnarray*}
\end{proof}

We now know that the Green's function $G_{\phi_T}$ is independent of the choice of $T\in A\setminus\FF_q$, and so we may define the \emph{Green's function associated to $\phi$} to be the function $G_{\phi}=G_{\phi_T}$, for any non-constant $T\in A$.    One important corollary of this is that the quantity
\[\lim_{|x|\to\infty}\left(\log|x|-G_{\phi_T}(x)\right)=\frac{1}{q^{r\deg(T)}-1}\log|a_{r\deg(T)}^{-1}|,\]
which would initially appear to depend on the choice of $T\in A\setminus\FF_q$, is in fact independent of this choice.  We will denote this quantity by $c(\phi)$.  Note that Theorem~4.1 of \cite{baker-hsia} shows that $c(\phi)$ is the logarithmic transfinite diameter of filled Julia set of $\phi/\CC_v$, where $\CC_v$ is the completion of the algebraic closure of $L$, which is also the logarithmic capacity of the Berkovich filled Julia set (see \cite{juliapaper} for more details).

We note two properties, also present in \cite{poonen}, which are easy to prove from the definitions.

\begin{lemma}\label{lem:greenspropertiesII}
For all $x, y\in L$ and all $a\in A$, we have
\[G_\phi(x+y)\leq \max\{G_\phi(x), G_\phi(y)\},\]
with equality unless $G_\phi(x)=G_\phi(y)$,
and
\[G_\phi(\phi_a(x))=|a|_\infty^rG_\phi(x).\]
\end{lemma}

Before proceeding to our discussion of local heights, we introduce two more pieces of notation, which also appear in \cite{juliapaper}.  Fix a basis $T_1, ..., T_m$ for $A$ as an $\FF_q$-algebra, as above.  For a Drinfeld module $\phi/L$, let $a_{i, j}\in L$ be defined by
\[\phi_{T_i}(x)=T_ix+a_{i, 1}x^q+\cdots+a_{i, r\deg(T_i)}x^{q^{t\deg(T_i)}},\]
for each $i$, and set
\[j_{\phi, v}=\max_{\substack{1\leq i\leq m\\1\leq j\leq r\deg(T_i)}}\left\{\frac{1}{q^j-1}\log|a_{i, j}|\right\}+c(\phi).\]
In the context of a global function field, $j_{\phi, v}$ will be the local contribution to the height of $j_\phi$.  We note, as proven in \cite{juliapaper}, that $j_{\phi, v}$ is non-negative, and that for finite places $v$ we have $j_{\phi, v}=0$ if and only if $\phi$ has potentially good reduction at $v$.
Similarly, for each non-constant $T\in A$, we define $j_{\phi_T, v}$ by
\[j_{\phi_T, v}=\max_{1\leq j\leq r\deg(T)}\left\{\frac{1}{q^j-1}\log|a_{j}|\right\}+c(\phi).\]
Finally, for any $\phi/L$ we define $B_T>0$ by
\[\log B_T=\log\max\{|\xi|:\xi\in\phi[T]\}+\frac{1}{q^{r\deg(T)}-1}\log^+|T^{-1}|.\]
We note that, by Lemma~\ref{lem:greensprops}, if $|x|>B_T$, we have $G_\phi(x)=\log|x|-c(\phi)$.

We end this section with a convenient lemma from \cite{juliapaper}.  In this lemma, and the next section, we will say that the valuation $v$ on $L$ is a \emph{finite place} if the images of the elements of $A$ in $L$ are all $v$-integral, and we will call $v$ an \emph{infinite place} otherwise.
\begin{lemma}
If $v$ is a finite place, then $j_{\phi, v}=j_{\phi_T, v}$, for any non-constant $T\in A$. 
\end{lemma}

\begin{proof}
The proof of this fact, which follows from an argument similar to that giving the Gauss Lemma, is found in \cite{juliapaper}.
\end{proof}


\subsection{The (local) minimal discriminant}
In this section we define a quantity $\Disc_{\phi, v}$, which will become the local contribution to the minimal discriminant $\Disc_{\phi/L}$ in Section~\ref{sec:global}.  We let $L$ be a complete, local $A$-field with valuation $v$, and assume additionally that the ring $\Ocal\subseteq L$ of integral elements is a normalized discrete valuation ring.  As usual, we will let $\CC_v\supseteq L$ be the completion of the algebraic closure of $L$.
\begin{definition}
We say that a Drinfeld module $\psi/L$ is an \emph{integral model of} $\phi/L$ if $\psi$ is $L$-isomorphic to $\phi$, and if $\psi_a(x)\in \Ocal[x]$ for every $a\in A$.  We define $\Disc_{\phi, v}$ by
\[\Disc_{\phi, v}=\min\{c(\psi):\psi/L\text{ is an integral model of }\phi\}.\]
We will call $\psi/L$ a \emph{minimal model} of $\phi/L$ if it is an integral model, and $\Disc_{\phi/L}=c(\psi)$.
\end{definition}
Note that, if $\pi$ is a uniformizer for $v$, then $\pi^{-N}\phi\pi^N$ is an integral model for $N$ sufficiently large, and so the set of which we are taking the minimum is not vacuous.  On the other hand, $v$ is discrete and $c(\psi)\geq 0$ for every integral model of $\phi$, so the minimum in the definition exists.
\begin{proposition}\label{prop:disclocal}
If $v$ is a finite place, then
\[j_{\phi, v}\leq \Disc_{\phi, v}< j_{\phi, v}+\deg(v).\]
In particular, if $j_{\phi, v}>0$ then \[j_{\phi, v}>\frac{1}{d+1}\max\{j_{\phi, v}+\Disc_{\phi, v}\}\] for some $d\geq 1$ depending only on $A$ and the rank of $\phi$.
\end{proposition}

\begin{proof}
Note that $j_{\phi, v}$ is a $\CC_v$-isomorphism invariant, and so if suffices to show that 
\[j_{\psi, v}\leq c(\psi)< j_{\psi, v}+\deg(v),\]
where $\psi$ is an integral model for $\phi/L$.  The first inequality is trivial from the definition of $j_{\psi, v}$, and the fact that the coefficients of $\psi_a(x)$ are integral for every $a\in A$.  Now suppose, toward a contradiction, that $c(\psi)\geq  j_{\psi, v}+\deg(v)$ and that $\pi$ is a uniformizer for $v$.  Then, if $T_1, ..., T_m$ is the generated set for $A$ used in the definition of $j_{\psi, v}$, we have $\frac{1}{q^j-1}\log|a_{i, j}|\leq -\deg(v)$,  for all $i$ and $j$, where
\[\psi_{T_i}(x)=T_ix+a_{i, 1}x^q+\cdots +a_{i, \deg(T_i)r}x^{q^{\deg(T_i)r}}.\]
In particular, if
\[\pi\psi\pi^{-1}_{T_i}(x)=T_ix+b_{i, 1}x^q+\cdots +b_{i, \deg(T_i)r}x^{q^{\deg(T_i)r}}\]
for all $i$, then $\frac{1}{q^j-1}\log|b_{i, j}|\leq 0$.  In other words, $\pi\psi\pi^{-1}$ is an integral model of $\phi$.  But $c(\pi\psi\pi^{-1})<c(\psi)$, contradicting the minimality of $\psi$.  It must be the case, then, that $\Disc_{\phi, v}< j_{\phi, v}+\deg(v)$.

Now, suppose that $j_{\phi, v}>0$.  Note that if $M=\max\{\deg(T_1), ...., \deg(T_m)\}$, where the $T_i$ are our fixed generators of $A$ as an $\FF_q$-algebra, and
\[d=d(A, r)=\lcm\{q^j-1:1\leq j\leq rM\},\]
then $j_{\phi, v}\in \deg(v)\ZZ[1/d]$.  In particular, $j_{\phi, v}=0$ or $j_{\phi, v}\geq \deg(v)/d$.   It follows, if $j_{\phi, v}\neq 0$, that
\[\Disc_{\phi, v}< j_{\phi, v}+\deg(v)\leq (1+d)j_{\phi, v},\]
whence the bound.
\end{proof}


\section{Local heights for Drinfeld modules}
\label{sec:localheights}

In this section we lay out a theory of local heights for Drinfeld modules.  Local heights for Drinfeld modules have previously been considered by Poonen \cite{poonen} and Ghioca \cite{ghioca1, ghioca2}, but we define them slightly differently here.  Throughout this section, $\CC_v$ will be a complete, algebraically closed $A$-field with valuation $v$.

\begin{definition}
Let  $\phi/\CC_v$ be a Drinfeld module.  Define the \emph{local height} associated to $\phi$ by
\[\lambda_{\phi}(x)=\log|x^{-1}|+G_{\phi}(x)+c(\phi).\]
\end{definition}

\begin{remark}
We note that $\lambda$ admits a natural extension to the Berkovich analytic space $\Aberk{\CC_v}$  associated to $\CC_v$ (see~\cite{baker-rumely}), namely
\[\lambda_{\phi}(x)=-\log\delta(x, 0)_\infty+G_{\phi}(x)+c(\phi),\]
where $\delta(x, y)_\infty$ is the Hsia kernel \cite[p.~73]{baker-rumely} and where $G_{\phi}$ is extended in the natural way.  In particular, one can check that this extension of $\lambda_\phi(x)$ agrees with $g_{\mu_\phi}(x, 0)$, where $g_{\mu_\phi}(x, y)$ is the Arakelov-Green's function on $\Pberk{\CC_v}\times\Pberk{\CC_v}$ for the invariant measure associated to $\phi_{T}(x)$, for any non-constant $T\in A$ \cite[p.~300]{baker-rumely}.  In some sense, then, $\lambda_\phi(x)$ is the correct local height from the capacity-theoretic viewpoint, relative to the identity element of the module.  Some of the properties of $\lambda_{\phi}$ outlined in Lemma~\ref{lem:lambdaprops} below may be deduced from the properties of $g_{\mu_\phi}$ derived in \cite{baker-rumely}, but we give direct and elementary proofs for completeness and simplicity.
\end{remark}

\begin{lemma}\label{lem:lambdaprops}
Let $\lambda_\phi$ be the local height defined above, for the Drinfeld module $\phi/\CC_v$.
\begin{enumerate}
\item The map $\lambda_\phi:\phi(\CC_v)\to \RR$ is continuous, except for a simple pole at the origin.
\item If $\psi\alpha=\alpha\phi$, for some $\alpha\in\CC_v^*$, then \[\lambda_\phi(x)=\lambda_\psi(\alpha x).\]
\item If $\phi$ has good reduction, then
\[\lambda_\phi(x)=\log^{+}|x^{-1}|.\]
\item If $\phi$ has potentially good reduction, then $\lambda_\phi$ is non-negative.
\item If $|x|>B_T$, for any $T\in A\setminus \FF_q$ then
\[\lambda_\phi(x)=0.\]
\item For any $a\in A$, we have
\[\lambda_\phi(\phi_a(x))=|a|_\infty^r\lambda_\phi(x)-\log\left|\frac{\phi_a(x)}{\Delta x^{|a|_\infty^r}}\right|,\]
where $\Delta$ is the leading coefficient of $\phi_a(x)$.
\end{enumerate}
\end{lemma}

\begin{remark}
Compare with Theorem~4.2 of \cite[p.~473]{ataec} for local heights relative to elliptic curves over $p$-adic fields.
\end{remark}

\begin{proof}[Proof of Lemma~\ref{lem:lambdaprops}]
\begin{enumerate}
\item This result is clear enough, but also follows from \cite[Proposition~8.66 p.~242]{baker-rumely}.
\item If $T\in A$ is non-constant, and \[\psi_T(x)=\alpha\phi_T(\alpha^{-1}z),\] we have
\[G_{\phi_T}(z)=G_{\psi_T}(\alpha z)\]
and $c(\psi)=c(\phi)-\log|\alpha|_v$, so
\begin{eqnarray*}
\lambda_{\phi}(x)&=&\log|x^{-1}|+G_{\phi_T}(x)+c(\phi)\\
&=&\log|x^{-1}|+G_{\psi_T}(\alpha x)+c(\psi)-\log|\alpha|\\
&=&\lambda_{\psi}(\alpha z).
\end{eqnarray*}
So the local height is coordinate invariant.

\item Note that if $\phi/\CC_v$ has good reduction, then the coefficients of $\phi_T$ are integral, and the leading coefficient is a unit.  It follows that $c(\phi)=0$, and from Lemma~\ref{lem:greensprops} that $G_{\phi_T}(x)=\log^+|x|$.
So \[\lambda_\phi(x)=\log|x^{-1}|+\log^+|x|=\log^+|x^{-1}|.\]

\item If $\phi$ has potentially good reduction then there is a $\psi/\CC_v$ with good reduction, and an $\alpha\in\CC_v^*$ such that $\psi\alpha=\alpha\phi$. By two of the claims above, we have \[\lambda_{\phi}(x)=\log^+|(\alpha x)^{-1}|\geq 0\] for all $x$.

\item 
This follows directly from Lemma~\ref{lem:greensprops}.

\item
This is proven by noting that
\begin{eqnarray*}
\lambda_\phi(\phi_a(x))&=&\log|\phi_a(x)^{-1}|+G_{\phi_a}(\phi_a(x))-\frac{1}{|a|_\infty^r-1}\log|\Delta|\\
&=&|a|_\infty^r\log|x^{-1}|+|a|_\infty^rG_{\phi_a}(x)-\frac{|a|_\infty^r}{|a|_\infty^r-1}\log|\Delta|-\log|\phi_a(x)|\\&&+|a|_\infty^r\log|x|+\frac{|a|_\infty^r-1}{|a|_\infty^r-1}\log|\Delta|\\
&=&|a|_\infty^r\lambda(x)-\log\left|\frac{\phi_a(x)}{\Delta x^{|a|_\infty^r}}\right|.
\end{eqnarray*}
\end{enumerate}
\end{proof}

We now proceed with some definitions and lemmas aimed at showing that one cannot have too large a set of points in $\phi(L)$ at which both $\lambda_\phi$ and $G_\phi$ take small values.
The following definition and lemmas are motivated by the work of Ghioca~\cite{ghioca1, ghioca2}.
\begin{definition}
If $\phi/L$ is a Drinfeld module, and $T\in A$ is non-constant, then we say that $x\in L$ is \emph{$T$-generic} if
\[|\phi_T(x)|=\max_{0\leq i\leq r\deg(T)}|a_ix^{q^i}|,\]
where $a_0=T$ as usual.  
\end{definition}

\begin{lemma}\label{lem:genericbound}
If $x\in \phi(\CC_v)$ is $T$-generic and
\[\log|\phi_T(x)|\leq \log B_T\]
then
\[\log|x|\leq c(\phi)+\frac{1}{(q^{r\deg(T)}-1)^2}\log^+|T^{-1}|.\]
\end{lemma}

\begin{remark}
Note that $c(\phi)$ is the average (logarithmic) size of the non-zero elements of $\phi[T]$, and so the lemma says that $T$-generic points whose image under $\phi_T$ is no larger than the largest elements of $\phi[T]$ are themselves no larger than the average elements of $\phi[T]$.
\end{remark}

\begin{proof}[Proof of Lemma~\ref{lem:genericbound}]
To begin,
write
\[\phi_T(x)=Tx+a_1x^q+\cdots+\Delta x^{q^{r\deg(T)}},\]
where we take $a_0=T$ and $a_{r\deg{T}}=\Delta$.  We will, for convenience, write $R=r\deg(T)$.
Using the hypothesis that $x$ is $T$-generic, we choose $i$ so that 
\[|\phi_T(x)|=|a_ix^{q^i}|=\max_{0\leq k\leq R}\left\{|a_kx^{q^k}|\right\}.\]
From the theory of Newton polygons, we see that we may also choose $0\leq j<R$ such that \[\frac{1}{q^R-q^j}\log|a_j/\Delta|+\frac{1}{q^R-1}\log^+|T^{-1}|=\log B_T.\]

We proceed much as in \cite{ghioca1}.  By hypothesis
\[q^i\log|x|+\log|a_i|\leq \log B_T=\frac{1}{q^R-q^j}\log|a_j/\Delta|+\frac{1}{q^R-1}\log^+|T^{-1}|.\]
Re-arranging this, we have
\[(q^{R+i}-q^{j+i})\log|x|+(q^R-q^j)\log |a_i|\leq \log |a_j|-\log|\Delta|+\frac{q^R-q^j}{q^R-1}\log^+|T^{-1}|.\]
Also, by hypothesis, we have
\[|a_ix^{q^i}|\geq |a_jx^{q^j}|,\]
and so we may conclude that
\[\log |a_j|\leq \log|a_i|+(q^i-q^j)\log|x|.\]
Combining these gives
\begin{multline*}
(q^{R+i}-q^{j+i})\log|x|+(q^R-q^j)\log |a_i|\\\leq  \log|a_i|+(q^i-q^j)\log|x|-\log|\Delta|+\frac{q^R-q^j}{q^R-1}\log^+|T^{-1}|
\end{multline*}
or
\begin{equation}\label{eq:ai}\left(q^{R+i}-q^{j+i}-q^i+q^j\right)\log|x|\leq (1-q^R+q^j)\log|a_i|-\log|\Delta|+\frac{q^R-q^j}{q^R-1}\log^+|T^{-1}|.\end{equation}
At the same time, our hypotheses also ensure that
\[|a_ix^{q^i}|\geq |\Delta x^{q^R}|,\]
and so
\[\log|a_i|\geq (q^R-q^i)\log|x|+\log|\Delta|.\]
But $q\geq 2$, and hence $q^j+1\leq q^{j+1}\leq q^R$, so we have $1-q^R+q^j\leq 0$. This gives 
\[(1-q^R+q^j)\log|a_i|\leq (1-q^R+q^j)\left((q^R-q^i)\log|x|+\log|\Delta|\right).\]
Combining this with \eqref{eq:ai} gives
\begin{multline*}
\left(q^{R+i}-q^{j+i}-q^i+q^j\right)\log|x|\\\leq (1-q^R+q^j)\left((q^R-q^i)\log |x|+\log|\Delta|\right)-\log|\Delta|+\frac{q^R-q^j}{q^R-1}\log^+|T^{-1}|,
\end{multline*}
whereupon
\[(q^R-1)(q^R-q^j)\log|x|\leq (q^R-q^j)\log|\Delta^{-1}|+\frac{q^R-q^j}{q^R-1}\log^+|T^{-1}|.\]
Recalling that $c(\phi)=\frac{1}{q^R-1}\log|\Delta^{-1}|$, 
this ensures that
\[\log|x|\leq c(\phi)+\frac{1}{(q^R-1)^2}\log^+|T^{-1}|.\]
\end{proof}

\begin{lemma}\label{lem:phiTsmall}
If $x\in\phi(\CC_v)$ is $T$-generic and \[\log|\phi_T(x)|\leq c(\phi)+\frac{1}{(q^{r\deg(T)}-1)^2}\log^+|T^{-1}|,\] then
\[-\log|x|_v+c(\phi)\geq (1-q^{-1})j_{\phi_T, v}-\frac{1}{q(q^{r\deg(T)}-1)^2}\log^+|T^{-1}|.\]
\end{lemma}

\begin{proof}
Suppose that $x$ is generic, and that $\log|\phi_T(x)|\leq c(\phi)$.  Then we have \[\log|\phi_T(x)|=\max_{0\leq i\leq r\deg(T)}\left\{ \log|a_ix^{q^i}|\right\}\geq \log|a_kx^{q^k}|,\] where $1\leq k\leq r\deg(T)$ is an index maximizing $\frac{1}{q^k-1}\log|a_k|$.  Note that, by definition, \[j_{\phi_T, v}=\frac{1}{q^k-1}\log|a_k|+c(\phi).\]  It follows that
\begin{eqnarray*}
c(\phi)+\frac{1}{(q^{r\deg(T)}-1)^2}\log^+|T^{-1}|&\geq&\log|\phi_T(x)|\\
&\geq & \log|a_k|+q^k\log|x|\\
&= & (q^k-1)(j_{\phi_T, v}-c(\phi))+q^k\log|x|,
\end{eqnarray*}
and hence that
\begin{eqnarray*}
-\log|x|_v+c(\phi)&\geq & -q^{-k}\left(c(\phi)+\frac{1}{(q^{r\deg(T)}-1)^2}\log^+|T^{-1}|-(q^k-1)(j_{\phi_T, v}-c(\phi))\right)+c(\phi)\\
&=&(1-q^{-k})j_{\phi_T, v}-q^{-k}\frac{1}{(q^{r\deg(T)}-1)^2}\log^+|T^{-1}|\\
&\geq&(1-q^{-1})j_{\phi_T, v}-\frac{1}{q(q^{r\deg(T)}-1)^2}\log^+|T^{-1}|,
\end{eqnarray*}
since $j_{\phi, v}$ and $\log^+|T^{-1}|$ are non-negative.
\end{proof}

The final lemma of this section shows, essentially, that we cannot have too large a collection of values $a\in A$ such that both $G_\phi(\phi_a(x))$ and $\lambda_\phi(\phi_a(x))$ are small.  This will help us establish our lower bounds in Section~\ref{sec:global}, since the sum of $G_\phi(x)$ over all places agrees with the same sum for $\lambda_\phi(x)$, as long as $x$ is non-zero.

\begin{lemma}\label{lem:localchoose}
Suppose that $X\subseteq A$ is an additive subgroup, and that for all $a\in X$ we have
\[\log|\phi_{T^2a}(x)|\leq \log B_T.\]
Then there is an additive subgroup $Y\subseteq X$ with $\#Y\geq q^{-4r^2\deg(T)^2} \# X$ such that for all $a\in Y$, we have either $\phi_a(x)=0$, or
\[-\log|\phi_a(x)|+c(\phi)\geq  (1-q^{-1})j_{\phi_T, v}-\frac{1}{q(q^{r\deg(T)}-1)^2}\log^+|T^{-1}|.\]
\end{lemma}

\begin{proof}
To begin, we note that it suffices to construct a sub\emph{set} $Y\subseteq X$ satisfying the criteria, since the inequality
\[-\log|x+y|\geq\min\{-\log|x|, -\log|y|\}\]
will ensure that the subgroup generated by $Y$ will also satisfy the criteria.  Second of all,  we note that it suffices to find a subset $Y\subseteq X$ of the appropriate size, such that for all $a\in Y$, both $\phi_a(x)$ and $\phi_{Ta}(x)$ are $T$-generic, since $\log|\phi_{T^2a}(x)|\leq \log B_T$ implies $\log|\phi_{Ta}(x)|\leq \log B_T$, and so we may apply Lemma~\ref{lem:genericbound} to $\phi_{Ta}(x)$ and then Lemma~\ref{lem:phiTsmall} to obtain the requisite bound on $|\phi_a(x)|$.

For any set $W\subseteq A$, we define $N_1(W)$ so that $q^{N_1(W)}$ is the number of elements $\xi\in\phi[T]$ such that $|\xi|\leq \max\{|\phi_a(x)|:a\in W\}$.
Note that, by examining the Newton polygon of $\phi_T(x)$, we see that $N_1(W)$ is always an integer, and $0\leq N_1(W)\leq R$ as long as $W$ is non-empty, where as usual we take $R=r\deg(T)$.
Similarly, define
$N_T(W)=N_1(TW)$, that is, $q^{N_T(W)}$ is the number of $\xi\in\phi[T]$ such that $|\xi|\leq \max\{|\phi_{Ta}(x)|:a\in W\}$.

For $\xi\in\phi[T]$, define
\[Z_\xi=\{y\in\phi(\CC_v):|y-\xi|<|\xi|=|y|\}\]
if $\xi\neq 0$, and
\[Z_0=\{y\in\phi(\CC_v):y\text{ is generic}\}.\]
Note that $\phi(\CC_v)=\bigcup_{\xi\in\phi[T]}Z_\xi$.
Let $X_1=X$, and suppose that for $i\geq 1$ we have constructed a set $X_i$.  For each $(\xi_1, \xi_2)\in \phi[T]\times\phi[T]$, let
\[X_{i, \xi_1, \xi_2}=\left\{a\in X_i:\phi_a(x)\in Z_{\xi_1}\text{ and }\phi_{Ta}(x)\in Z_{\xi_2}\right\}.\]
Clearly the sets $X_{i, \xi_1, \xi_2}$ cover $X_i$, and there are $(q^R)^2$ of them, so there is some pair $(\xi_1, \xi_2)$ such that $\# X_{i, \xi_1, \xi_2}\geq q^{-2R}\# X_i$.

We consider several cases.

\noindent\textbf{Case 1:} $\xi_1\neq 0$.

Choose some $a\in X_{i, \xi_1, \xi_2}$, and let $X_{i+1}=X_{i, \xi_1, \xi_2}-a$, so that $\# X_{i+1}=\# X_{i, \xi_1, \xi_2}$.  Since
\[\max\{|\phi_{Tb}(x)-\phi_{Ta}(x)|:b\in X_{i, \xi_1, \xi_2}\}\leq\max\{|\phi_{Tb}(x)|:b\in X_i\},\]
we see that $N_T(X_{i+1})\leq N_T( X_{i})$.  On the other hand,
\[|\phi_b(x)-\phi_a(x)|=|\phi_b(x)-\xi_1+\phi_a(x)-\xi_1|<|\xi_1|=|\phi_b(x)|=|\phi_a(x)|\]
for all $a, b\in X_{i, \xi_1, \xi_2}$, by definition.  It follows that
\[N_1(X_{i+1})<N_1(X_{i, \xi_1, \xi_2})\leq N_1(X_i).\]
In particular, since $N_1(W)$ is a non-negative integer for any non-empty $W\subseteq A$, the present case can arise for at most $R$ values of $i$.

\noindent\textbf{Case 2:} $\xi_1=0$, $\xi_2\neq 0$.

In this case we again choose some $a\in X_{i, \xi_1, \xi_2}$, and let $X_{i+1}=X_{i, \xi_1, \xi_2}-a$.  Arguments just as in Case 1 show that here $N_1(X_{i+1})\leq N_1(X_i)$, while $N_T(x_{i+1})<N_t(X_i)$.  In particular, this case can also arise for at most $R$ values of $i$.

\noindent\textbf{Case 3:} $\xi_1=\xi_2=0$.

In this case, we may take $Y=X_{i, 0, 0}$, since then $\phi_a(x)$ and $\phi_{Ta}(x)$ are both $T$-generic for all $a\in Y$, by construction.  We then have that $\# Y\geq q^{-2Ri}\# X$, and we have seen that we arrive in this case with $i\leq 2R$.  This proves the lemma.
\end{proof}

Before proceeding to the next section, which contains the proof of the main results, we give an explicit bound on the difference between the canonical local height $\lambda_\phi(x)$ and the naive local height $\log^+|x^{-1}|$.  This will be used below to estimate the difference between the naive and canonical global heights.

\begin{lemma}\label{lem:localheightdiff}
Let $\phi/\CC_v$ be a Drinfeld module of rank $r$.
We have, for all $x\in \CC_v$, and any non-constant $T\in A$,
\begin{multline}\label{eq:localheightdiff}
-\frac{|T|_\infty^r}{(|T|_\infty^r-1)^2}\log^+|T|-\frac{|T|_\infty^r}{|T|_\infty^r-1}j_{\phi_T, v}-\max\{0, -c(\phi)\}\\
\leq \lambda_\phi(x)-\log^+|x^{-1}|\\
\leq j_{\phi_T, v}+\frac{1}{|T|_\infty^r-1}\log^+|T|+\max\{0, c(\phi)\}
.\end{multline}
\end{lemma}

\begin{proof}
Write
\[\phi_T(x)=Tx+a_1x^q+\cdots+\Delta a^{q^{r\deg(T)}},\]
consider first the case in which $\Delta=1$, and write $R=r\deg(T)$ as usual.
By the triangle inequality, we have
\begin{eqnarray*}
\log|\phi_T(x)|&\leq& \log\max\left\{|Tx|, |a_1x^q|, ..., |x^{q^R}|\right\}\\
&\leq& \log\max\{|T|, |a_1|, ...,1\}+q^R\log^+|x|\\
&\leq& \log\max\{|a_1|, ..., 1\}+\log^+|T|+q^R\log^+|x|\\
&\leq& (q^R-1)j_{\phi_T, v}+\log^+|T|+q^R\log^+|x|
\end{eqnarray*}
It follows by induction that
\[G_\phi(x)- \log^+|x|\leq j_{\phi_T, v}+\frac{1}{q^R-1}\log^+|T|.\]

On the other hand, we have
\[\log B_T=\max_{0\leq i\leq R}\left\{0, \frac{1}{q^R-q^i}\log|a_i|\right\}\leq \frac{1}{q^R-1}\log^+|T|+j_{\phi_t, v},\]
since $q\geq 2$.  If $\log|x|>\log B_T$, then we have
\[q^{-R}\log|\phi_T(x)|=q^{-R}\log| x^{q^R}|=\log|x|.\]

If, on the other hand, $\log|x|\leq \log B_T$, then we have
\[q^{-R}\log^+|\phi(x)|- \log^+|x|\geq -\log B_T\geq -\frac{1}{q^R-1}\log^+|T|-j_{\phi_T, v}.\]
By induction, we have
\begin{eqnarray*}
q^{-nR}\log^+|\phi_T(x)|-\log^+|x|&\geq& -\left(1+\frac{1}{q^R}+\cdots\right)\left(\frac{1}{q^R-1}\log^+|T|+j_{\phi_T, v}\right)\\
&=&-\frac{q^R}{(q^R-1)^2}\log^+|T|-\frac{q^R}{q^R-1}j_{\phi_T, v}.
\end{eqnarray*}
So we have
\[-\frac{q^R}{(q^R-1)^2}\log^+|T|-\frac{q^R}{q^R-1}j_{\phi_T, v}\leq G_\phi(x)- \log^+|x|\leq j_{\phi_T, v}+\frac{1}{q^R-1}\log^+|T|.\]

Now, if $\Delta\neq 1$, we may apply the above bound to obtain a result for $\psi=\alpha \phi\alpha^{-1}$, where $\alpha^{q^r-1}=\Delta$.  Replacing $x$ with $\alpha x$ in the above, 
and noting that $G_\phi(x)=G_{\psi}(\alpha x)$, $j_{\phi_T, v}=j_{\psi_T, v}$ and $c(\phi)=-\log|\alpha| $
we have
\begin{multline*}
-\frac{q^R}{(q^R-1)^2}\log^+|T|-\frac{q^R}{q^R-1}j_{\phi_T, v}-\log^+|\alpha^{-1}|\\\leq G_\phi(x)- \log^+| x|\\\leq j_{\phi_T, v}+\frac{1}{q^R-1}\log^+|T|+\log^+|\alpha|.
\end{multline*}
from the inequality
\[\log^+|x|-\log^+|\alpha^{-1}|\leq \log^+|\alpha x|\leq \log^+|x|+\log^+|\alpha|.\]
We now note that $\log|\alpha|=-c(\phi)$, and that
\[\lambda_\phi(x)-\log^+|x^{-1}|=G_\phi(x)-\log^+|x|+c(\phi),\]
and hence we have \eqref{eq:localheightdiff}.
\end{proof}

In Lemma~\ref{lem:lambdaprops} we saw that $\lambda_{\phi}(x)$ is non-negative in the case of good reduction.  To establish Theorem~\ref{th:jplaces}, however, we will need a stronger statement in the case of bad reduction which is potentially good, since in this case we have $\Disc_{\phi, v}>0$.
\begin{lemma}\label{lem:potgoodred}
Suppose that $v$ is a finite place, that $\phi/L$ has potentially good reduction, and that $\Disc_{\phi, v}>0$.  Then we have
\[\lambda_\phi(x)+G_{\phi}(x)\geq\frac{1}{d-1} \Disc_{\phi, v}\]
for all $x\in \phi(L)$, where $d(A, r)$ is as defined in Proposition~\ref{prop:disclocal}.
\end{lemma}

\begin{proof}
Assume, without loss of generality, that $\phi$ is a minimal integral model of itself, and choose an $\alpha\in\CC_v$ such that $\psi=\alpha\phi\alpha^{-1}$ has good reduction.  Note that, in this case, $\Disc_{\phi/L}=c(\phi)=-\log|\alpha|$, since $c(\psi)=0$. Furthermore, since $\psi$ has good reduction, we see that $j_{\phi, v}=j_{\psi, v}=0$, and hence by Proposition~\ref{prop:disclocal} it must hold that $0< v(\alpha)<1$ (the case $v(\alpha)=0$ is excluded by our hypotheses).

  From the proof of Lemma~\ref{lem:lambdaprops}, we see that
$\lambda_{\phi}(x)=\lambda_{\psi}(\alpha x)$
for all $x\in L$, and $G_\phi(x)=G_{\psi}(\alpha x)$.  First suppose that $|\alpha x|\leq 1$.  Then, since $v(x)\in\ZZ$ and $0< v(\alpha)<1$, we have
\[\lambda_{\phi}(x)=\lambda_{\psi}(\alpha x)=\log^+|\alpha^{-1}x^{-1}|=(v(\alpha)+v(x))\deg(v)\geq v(\alpha)\deg(v)=\Disc_{\phi, v}.\]
The claimed inequality follows from the fact that $G_\phi(x)\geq 0$ for all $x$.

On the other hand, suppose that $|\alpha x|>1$, and note
that $0< v(\alpha)<1$ is a rational number with denominator at most $d=d(A, r)$.  It follows, since $v(x)$ is an integer for $x\in L$, that $v(\alpha x)<0$ implies
\[v(\alpha x)\leq v(\alpha)-1\leq -\frac{1}{d}\leq -\frac{1}{d-1}v(\alpha).\]
Since $\psi$ has good reduction, we obtain
\[G_\phi(x)=G_\psi(\alpha x)=\log^+|\alpha x|\geq \frac{1}{d-1}\log|\alpha|=\frac{1}{d-1}\Disc_{\phi, v}.\]
\end{proof}

Finally, we prove a lower bound on the Green's function associated to $\phi$, to be used in Section~\ref{sec:global}.

\begin{lemma}\label{lem:greenslower}
If $|x|>B_{T}$, then
\[G_{\phi}(x)\geq \left(\frac{q-1}{q^{r\deg(T)}-1}\right)j_{\phi_T, v}.\]
\end{lemma}
\begin{proof}
It follows from Lemma~\ref{lem:greensprops}~(iii) that for all $x$ with $|x|>B_{T}$,  we have
\[G_{\phi}(x)=\log|x|-c(\phi).\]
Note that, by changing coordinates, it suffices to prove the statement in the case where $\Delta=1$.  Then we have, for $|x|>\log B_{T}$,
\begin{eqnarray*}
G_{\phi}(x)&=&\log|x|-c(\phi)\\
&> & \max\left\{\frac{1}{q^{r\deg(T)}-1}\log|T|, \frac{1}{q^{r\deg(T)}-q}\log|a_1|, ..., \frac{1}{q^{r\deg(T)}-q^{r\deg(T)-1}}\log|a_{r\deg(T)-1}|, 0\right\}\\
&\geq&\max\left\{\frac{1}{q^{r\deg(T)}-q}\log|a_1|, ..., \frac{1}{q^{r\deg(T)}-q^{r\deg(T)-1}}\log|a_{r\deg(T)-1}|, 0\right\}\\
&\geq &\frac{q-1}{q^{r\deg(T)}-q}\max\left\{\frac{1}{q-1}\log|a_1|, ..., \frac{1}{q^{r\deg(T)-1}-1}\log|a_{r\deg(T)-1}|, 0\right\}\\
&=&\left(\frac{q-1}{q^{r\deg(T)}-1}\right)j_{\phi_T, v}.
\end{eqnarray*}
\end{proof}


\section{Drinfeld modules over global fields}
\label{sec:global}

Throughout this section, $L$ is a global $A$-field, as defined in Section~\ref{sec:prelims}, and $M_L$ is its set of places.
For any place $v\in M_L$, we will denote by $L_v$ the completion of $L$.  If $\phi/L$ is a Drinfeld module, then $\phi$ naturally gives rise to a Drinfeld module over $L_v$, which we will also call $\phi$.  The objects $c(\phi)$, $B_T$, $G_\phi$, and $\lambda_\phi$ defined in Section~\ref{sec:localfields} for $\phi/L_v$ will now be denoted by $c_v(\phi)$, $B_{T, v}$, $G_{\phi, v}$ and $\lambda_{\phi, v}$, to emphasize the dependence on the particular place $v\in M_L$, and similarly for finite extensions.  We recall the canonical height defined by Denis \cite{denis}, which decomposes (as per Poonen \cite{poonen}) as a sum of Green's functions.
\[\hat{h}_\phi(x)=\sum_{v\in M_{L'}}\frac{[L'_v:L_v]}{[L':L]}G_{\phi, v}(x).\]
  Note that it follows immediately from the product formula that
\begin{equation}\label{eq:hhatdef}\hat{h}_\phi(x)=\sum_{v\in M_L}\frac{[L'_v:L_v]}{[L':L]}\lambda_{\phi, v}(x),\end{equation}
so long as $x\neq 0$.  Again, this gives a well-defined function $\hat{h}_\phi:\phi(\reas{L})\to \RR$.

  In his original construction of the canonical height associated to a Drinfeld $A$-module, Denis \cite{denis} showed the this quantity differs from the Weil height by only a bounded amount.  The purpose of the first result of this section is to make this bound explicit, in terms of various quantities related to $\phi$
We define, for a finite extension $L'/L$ and a Drinfeld module $\phi/L'$,
\[h(\phi)=h(j_\phi)+\sum_{v\in M_{L'}}\frac{[L'_v:L_v]}{[L':L]}\max\{0, c_v(\phi)\}.\]
Note that if $T\in A$ is any non-constant element, and
\[\phi_T(x)=Tx+\cdots+a_{r\deg(T)}x^{q^{r\deg(T)}},\]
then we have
\[\sum_{v\in M_{L'}}\frac{[L'_v:L_v]}{[L':L]}\max\{0, c_v(\phi)\}=\frac{1}{|T|_\infty^r-1}h(a_{r\deg(T)}),\]
and so $h(\phi)$ is just the sum of the heights of two quantities related to $\phi$.  To justify using the terminology of heights, we prove a Northcott-type result for this quantity.
\begin{proposition}\label{prop:northcott}
For any quantities $B$ and $D$, there are only finitely many Drinfeld modules $\phi/E$ such that $E/L$ is a finite extension of degree at most $D$, and $h(\phi)\leq B$.
\end{proposition}

\begin{proof}
First, note that if $h(\phi)\leq B$ and $\phi$ is defined over an extension of degree at most $D$ of $L$, then $j_\phi$ is a point of bounded height and bounded algebraic degree, over $L$.  In particular, there can be only finitely many $\overline{L}$-isomorphism classes of Drinfeld modules satisfying $h(\phi)\leq B$ defined over extensions of degree at most $D$.  It suffices, then, to show that each such isomorphism class contains only finitely many instances $\phi$ with $h(\phi)\leq B$, and defined over extensions of degree at most $D$.  But if $\phi$ is such an instance, and
\[\phi_{T_i}(x)=T_ix+a_{i, 1}x^q+\cdots +a_{i, \deg(T_i)r}x^{q^{\deg(T_i)r}},\]
then $h(\phi)\leq B$ implies $h(a_{i, r\deg(T_i)})\leq (q^{r\deg(T_i)}-1)B$ for each $i$.  In particular, the elements $a_{i, r\deg(T_i)}$ reside in a set of bounded height and degree, for each $i$.  So there are only finitely many choices for each $i$.  But suppose that $\phi \alpha=\alpha \psi$ for some $\alpha\in\overline{L}$, and that
\[\phi_{T_i}(x)=T_ix+b_{i, 1}x^q+\cdots +a_{i, \deg(T_i)r}x^{q^{\deg(T_i)r}},\]
for all $i$.  Then $\alpha^{q\deg(T_i)-1}=1$ for all $i$.  In particular, there are only finitely many pairwise isomorphic $\phi/\overline{L}$ corresponding to a given choice of $j_{\phi, v}$ and a given choice of $a_{1, q\deg(T_1)}, ..., a_{m, q\deg(T_m)}$.
\end{proof}

We now state a slightly stronger version of Theorem~\ref{th:zimmer}, recalling that $L^{\mathrm{Sep}}\subseteq \reas{L}\subseteq \overline{L}$, with equality in the second inclusion when $L=K$, for example.
\begin{theorem}\label{th:heightdiff}
Let $L'/L$ be a finite extension contained in $\reas{L}$, and let $\phi/L'$ be a Drinfeld module of rank $r\geq 1$.  There exist constants $C_1$ and $C_2$, depending only on $A$, $L$, and $r$ such that
 for all $x\in \reas{L}$, we have
\[-C_1-\left(\frac{1}{1-q^{-r}}\right)h(\phi)\leq \hat{h}_\phi(x)-h(x)\leq h(\phi)+C_2.\]
\end{theorem}

\begin{proof}
Let $T_1, ..., T_m$ generate $A$ as an $\FF_q$-algebra.
By Lemma~\ref{lem:localheightdiff}, we have the inequality \eqref{eq:localheightdiff} for each $T_i$, for each place $v\in L'$.  Note that if $v$ is a finite place, then $j_{\phi_{T_i}, v}=j_{\phi, v}$ for each $i$, and $\log^+|T_i|_v=0$, and so we have
\[
-\frac{1}{1-|T_i|_\infty^{-r}}j_{\phi, v}-\max\{0, -c_v(\phi)\}
\leq \lambda_\phi(x)-\log^+|x^{-1}|
\leq j_{\phi, v}+\max\{0, c_v(\phi)\}
\]
for each such place.  Note that, since $j_{\phi, v}\geq 0$, we may weaken the lower bound slightly by replacing $1/(1-|T|_\infty^{-r})$ by $1/(1-q^{-r})$, and similarly for the term $\max\{0, -c_v(\phi)\}$.  Thus we have
\[
-\frac{1}{1-q^{-r}}\left(j_{\phi, v}+\max\{0, -c_v(\phi)\}\right)
\leq \lambda_\phi(x)-\log^+|x^{-1}|
\leq j_{\phi, v}+\max\{0, c_v(\phi)\}
\]
at every finite place $v$.  Note that if $L$ has no infinite places, we may sum this over $v\in M_{L}$ to obtain the claimed inequality with $C_1=C_2=0$.

  At every infinite place, we similarly have
\begin{multline*}
-\frac{|T_i|_\infty^r}{(|T_i|_\infty^r-1)^2}\log^+|T_i|_v-\frac{1}{1-q^{-r}}\left(j_{\phi_{T_i}, v}+\max\{0, -c_v(\phi)\}\right)\\
\leq \lambda_\phi(x)-\log^+|x^{-1}|\\
\leq j_{\phi_{T_i}, v}+\max\{0, c_v(\phi)\}+\frac{1}{|T_i|_\infty^r-1}\log^+|T_i|_v
.\end{multline*}
Since $j_{\phi, v}=\max_{1\leq i\leq m } j_{\phi_{T_i}, v}$, summing over places gives the claimed bound with
\[C_1=\sum_{i=1}^m\frac{|T_i|_\infty^r}{(|T_i|_\infty^r-1)^2}\sum_{v\in  M_{L}^\infty}\log^+|T_i|_v\]
and
\[C_2=\sum_{i=1}^m\frac{1}{|T_i|_\infty^r-1}\sum_{v\in M_L}\log^+|T_i|_v\]

\end{proof}

\begin{remark}
Note that the constants $C_1$ and $C_2$ depend on $L$ only because of the generality in which we are working.  If we consider only those $L$ which are finite extensions of $K$, it makes sense to consider all heights relative to $K$ (in other words, to take $K$ as the ground field).  In this case the constants above depend only on $A$ and $r$.

If, on the other hand, $L/K$ is transcendental, then every place is a finite place, and the constants $C_1$ and $C_2$ vanish.
\end{remark}

Note that if $X/L$ is a curve, and $\phi/L(X)$ a Drinfeld module, then both $j_\phi$ and $a_{r\deg(T)}$, in the notation above, can be viewed as morphisms from $X$ to $M_{A, r}$ and $\PP^1$, respectively.  It follows that $h(\phi_t)=O(h(t))$.

\begin{corollary}
Let $X/L$ be a curve, let $h$ be a height on $X$ with respect to an ample divisor, and let $\phi/L(X)$ be a Drinfeld $A$-module.  Then we have
\[\hat{h}_{\phi_t}(x)=h(x)+O(h(t)),\]
where the implied constant depends only on the generic fibre $\phi$ and the choice of height $h$.
\end{corollary}

We now define the quantity $\Disc_{\phi/L}$ alluded to in the introduction, which is similar to the minimal discriminant used by Taguchi~\cite{taguchi}.
\begin{definition}
Let $\phi/L$ be a Drinfeld module and, for each $v\in M_L^0$, let $\Disc_{\phi, v}$ be the (local) minimal discriminant defined above.  Then the \emph{minimal discriminant of $\phi$} is the formal $\QQ$-linear combination of elements of $M_L^0$ given by
\[\Disc_{\phi, v}=\sum_{v\in M_L^0}\Disc_{\phi, v}[v].\]
We will define the \emph{discriminant} of $\phi/L$ to be the quantity
\[\Delta_{\phi/L}=\sum_{v\in M_L^0}c_v(\phi)[v].\]
We will say that $\phi$ is a \emph{global minimal model} if $\Delta_{\phi, v}=\Disc_{\phi, v}$, and that $\phi$ is \emph{quasi-minimal} if $\phi$ is a $v$-integral model for every $v\in M_L^0$, and  $\deg(\Delta_{\phi/L})=\sum_{v\in M_L^0} c_v(\phi)$ is minimal within the $L$-isomorphism class of $\phi$.
\end{definition}

Note that, if $B\subseteq L$ is the ring of elements integral at every infinite place, then $\Disc_{\phi/L}$ and $\Delta_{\phi/L}$ can be interpreted as divisors on $\Spec(B)$ with coefficients in $\QQ$, and we will use this terminology.  It is clear from the definition that $\Disc_{\phi/L}\leq \Delta_{\phi/L}$ if $\phi$ is defined over $B$, so minimal models are always quasi-minimal, but the converse might not hold.  Indeed, suppose that $\phi/L$ is a Drinfeld module, and for each $v\in M_L^0$ choose an $\alpha_v\in L_v$ so that $\alpha_v\phi\alpha_v^{-1}$ is a (local) minimal model.  Then $\phi/L$ has a global minimal model if and only if there is a $\beta\in L$ such that $|\beta|_v=|\alpha_v|_v$ for each finite place $v$.  In other words, if we define the \emph{Weierstrass divisor of $\phi/L$} to be the divisor \[\mathfrak{a}_{\phi/L}= \sum_{v\in M_L^0}v(\alpha_v)[v]\] on $\Spec(B)$, then the class of $\mathfrak{a}_{\phi/L}$ is an $L$-isomorphism invariant of $\phi$, and is trivial if and only if $\phi$ admits a global minimal model over $L$.

Although not every Drinfeld module over every field admits a minimal model, we see below that, at least for finite extensions of $K$, quasi-minimal models are close to being minimal.

\begin{proposition}\label{lem:classgp}
For any finite extension $L/K$ and any quasi-minimal $\phi/L$, we have
\[\deg(\Disc_{\phi/L})\leq \deg(\Delta_{\phi/L})\leq \deg(\Disc_{\phi/L})+g(L)+[L:K]-1,\]
where $g(L)$ is the genus of $L$.
\end{proposition}

\begin{proof}
Note that $\Delta_{\phi/L}=\Disc_{\phi/L}+\mathfrak{a}_{\phi/L}$, where $\mathfrak{a}_{\phi/L}$ is the Weierstrass divisor defined above.  The first inequality is obvious, since $\mathfrak{a}_{\phi/L}\geq 0$, and for the second inequality we wish to bound $\deg(\mathfrak{a}_{\phi/L})$ given that $\phi$ is quasi-minimal. 
 Note that if $\beta\in L$ and $\psi\beta = \beta\phi$, then we have $\mathfrak{a}_{\phi/L}=\mathfrak{a}_{\psi/L}+(\beta)_\mathrm{fin}$, where
\[(\beta)_\mathrm{fin}=\sum_{v\in M_L^0}v(\beta)[v],\]
and $\Delta_{\phi/L}=\Delta_{\psi/L}+(\beta)_\mathrm{fin}$.

Let $v\in M_L^\infty$ be any infinite place, and suppose that $\beta\in L\setminus\{0\}$ satisfies $(\beta)\leq\mathfrak{a}_{\phi/L}-[v]$, where $(\beta)$ is the usual divisor associated to $\beta$.  Then $(\beta)_\mathrm{fin}\leq \mathfrak{a}_{\phi/L}$, since the latter divisor is supported on $M_L^0$, and so if $\psi=\beta\phi\beta^{-1}$, we see that \[\Delta_{\psi/L} =\Disc_{\phi/L}+\mathfrak{a}_{\phi/L}-(\beta)_\mathrm{fin}\geq \Disc_{\phi/L}.\]
In particular, $\psi/L$ is an integral model, and $\deg(\Delta_{\phi/L})=\deg(\Delta_{\psi/L})+\deg(\beta)_\mathrm{fin}$.  By the quasi-minimality of $\psi$, we must have $\deg(\beta)_\mathrm{fin}\leq 0$.  But we also have $(\beta)-(\beta)_\mathrm{fin}\leq -[v]$, and so 
\[0=\deg((\beta)_\mathrm{fin})+\deg((\beta)-(\beta)_\mathrm{fin})\leq -\deg(v)<0,\]
which is impossible.

In other words, we have shown that there is no element $\beta\in L\setminus\{0\}$ such that $(\beta^{-1})+\mathfrak{a}_{\phi/L}-[v]\geq 0$, and so the Riemann-Roch space $\mathscr{L}(\mathfrak{a}_{\phi/L}-[v])$ is trivial.  It follows that 
\[\deg(\mathfrak{a}_{\phi/L})\leq g(L)-1+\deg(v)\leq g(L)+[L:K]-1,\]
where $g(L)$ is the genus of $L$.
\end{proof}


We may now commence with the technical lemmas needed for the proof of Theorem~\ref{th:jplaces}.  
If $b\in A$ and $C\geq 0$, the \emph{$C$-truncated ideal generated by $b$} will be the set
\[I(b, C)=b\cdot\{a\in A:|a|_\infty\leq C\}.\]

\begin{lemma}\label{lem:basinornot}
Fix $x\in \phi(L)$,  $B\geq 1$, $b\in A$, and let $S\subseteq M_L$ be a finite set of places.  Then there exists a $B'\geq 1$ and $b'\in A$ with $I(b', B')\subseteq I(b, B)$  such that for each $v\in S$ we have either
 \begin{enumerate}
 \item for all $a\in I(b', B')$, $|\phi_{T^2a}(x)|_v\leq B_{T, v}$ or 
 \item for all $a\in I(b', B')$, $|\phi_{a}(x)|_v>B_{T, v}$ or $a=0$.
 \end{enumerate}
 Furthermore we may take $B'\geq B^{2^{-\# S}}|T|_\infty^{-2}$
\end{lemma}

\begin{proof}
Let $b_0=b$, let $B_0=B$, and let $C$ be a constant to be determined later, so that
\[I(b_0, B_0)=I(b, B)=b\cdot \{a\in A:|a|_\infty\leq B\}.\]
Now, order the places $S=\{v_1, v_2, ..., v_{s}\}$, where $s=\# S$.  

Let $i\geq 0$, and suppose that for all $a\in I(b_i, B_i)$ with $|a|_\infty\leq |b_i|_\infty  C B_i^{1/2}$ we have $|\phi_{T^2a}(x)|_v\leq B_{T, v}$.  Then we set $b_{i+1}=b_i$ and $B_{i+1}=C B_i^{1/2}\leq B_{i}$.  Note that $I(b_{i+1}, B_{i+1})\subseteq I(b_i, B_i)$, and for all $a\in I(b_{i+1}, B_{i+1})$ we have $|\phi_{T^2a}(x)|_v\leq B_{T, v}$.

If this is not the case, then there exists an $a\in I(b_i, B_i)$ with $|a|_\infty\leq |b_i|_\infty CB_i^{1/2}$ and $|\phi_{T^2a}(x)|>B_{T, v}$.  Then for all $d\in T^2a A$ we have either $d=0$, or $|\phi_d(x)|_v>B_{T, v}$.  Let $a_{i+1}=T^2a$ and $B_{i+1}=CB_i^{1/2}$.
Note that by construction, for all $d\in I(b_{i+1}, B_{i+1})$, either $d=0$ or else $|\phi_d(x)|_v>B_{T, v}$.  Also, note that $I(b_{i+1}, B_{i+1})\subseteq I(b_i, B_i)$, since $b_i\mid b_{i+1}$ and if $d=b_{i+1} c$, with $|c|_\infty\leq B_{i+1}$, then $d=b_i\cdot \frac{T^2a}{b_i}\cdot c$ with
\[\left|b_i\frac{T^2a}{b_i}c\right|_\infty\leq |b_i|_\infty |T|_\infty^2 C B_i^{1/2}CB_i^{1/2}\leq |b_i|_\infty B_i,\]
as long at $C=|T|_\infty^{-1}$, and so $d\in I(b_i, B_i)$.

Note that in either case, $B_{i+1}=|T|_\infty^{-1} B_i^{1/2}$, and so we have
\[B'=B_s\geq B_0^{2^{-s}}|T|_\infty^{-(1+\frac{1}{2}+\cdots+\frac{1}{2^{s-1}})}\geq B^{2^{-s}}|T|_\infty^{-2}.\]
\end{proof}

We are now in a position to prove Theorem~\ref{th:jplaces}.  We prove the following slightly more general result, and then prove Theorem~\ref{th:jplaces} below as a corollary.
\begin{theorem}\label{th:generaljplaces}
Let $L$ be a global $A$-field, and let $s, r\geq 1$.  Then there are constants $\epsilon>0$ and $C$, and an ideal $\mathfrak{a}\subseteq A$, such that for every Drinfeld module $\phi/L$ of rank $r$ with genuinely bad reduction at at most $s$ places, we have either $x\in \phi[\mathfrak{a}]$ or else
\[\hat{h}_\phi(x)\geq \epsilon\max\{h(j_\phi), \deg(\Disc_{\phi/L})\}-C\]
for each $x\in \phi(L)$.
\end{theorem}

\begin{proof}
Let $x\in \phi(L)$, let $S$ be a set containing all of the places of persistently bad reduction of $\phi/L$, including all infinite places, and fix a set of (non-constant) generators $T_1, ..., T_m$ for $A$ as an $\FF_q$-algebra.  Note that, by Lemma~\ref{lem:lambdaprops} above we have $\lambda_{\phi, v}(x)\geq 0$ for all $v\not\in S$, and all non-zero $x\in \phi(L)$.  By Lemma~3.7 of \cite{juliapaper}, we also see that if $v\in M_L^0$ and $\phi$ has potentially good reduction at $v$, then $j_{\phi, v}=0$.  In other words, for any $\epsilon>0$ we have
\[\lambda_{\phi, v}(x)+G_{\phi, v}(x)\geq 0 = \epsilon j_{\phi, v}\]
for any $v\not\in S$, and for any $x\in \phi(L)$.  On the other hand, by Lemma~\ref{lem:potgoodred} we also have
\[\lambda_{\phi, v}(x)+G_{\phi, v}(x)\geq \frac{1}{d(A, r)-1}\Disc_{\phi, v}\]
for all places of potentially good reduction, and so 
\[\lambda_{\phi, v}(x)+G_{\phi, v}(x)\geq 0 = \epsilon \max\{j_{\phi, v}, \Disc_{\phi, v}\}\]
so long as $\epsilon\leq 1/(d(A, r)-1)$.
  The challenge is to obtain a lower bound similar to this at places $v\in S$.

We will fix some large value $C_m$, to be specified later, which will depend only on $m$, $r$, and $s=\# S$.  Applying Lemma~\ref{lem:basinornot} to $T_1$ and to the set $I(1, C_m)$, we find a subset $I(b_1, C_{m-1})$ such that for all $v\in S$ we have either $|\phi_{T_1^2a}(x)|_v\leq B_{T_1, v}$ for all $a\in I(b_1, C_{m-1})$,  or we have $|\phi_a(x)|_v>B_{T_1, v}$ for all non-zero $a\in I(b_1, C_{m-1})$.  Also, note that $C_{m-1}\geq C_m^{2^{-s}}|T_1|_\infty^{-2}$.

Now, construct $I(b_2, C_{m-2})\subseteq I(B_1, C_{m-1})$ such that
\[C_{m-2}\geq C_{m-1}^{2^{-s}}|T_2|_\infty^{-2}\geq C_m^{2^{-2s}}|T_1|_\infty^{-2^{1-s}}|T_2|_\infty^{-2}\]
and such that for each $v\in S$ we have either $|\phi_{T_2^2a}(x)|_v\leq B_{T_2, v}$ for $a\in I(b_{2}, C_{m-2})$ or we have $|\phi_a(x)|_v> B_{T_2, v}$ for all non-zero $a\in I(b_2, C_{m-2})$.  Proceeding inductively, we find a $b_m\in A$ and $C_0$ such that for each $1\leq i\leq m$ and for each $v\in S$, we have either $|\phi_{T_i^2a}(x)|_v\leq B_{T_i, v}$ for $a\in I(b_m, C_0)$ or we have $|\phi_a(x)|_v> B_{T_i, v}$ for all non-zero $a\in I(b_m, C_0)$.  We will denote by $S_{1, i}$ the set of places at which the first phenomenon occurs, and $S_{2, i}$ the set at which the second does, so that $S=S_{1, i}\cup S_{2, i}$ for any $1\leq i\leq m$.  Note also that $C_0\geq C_m^{2^{-m s}} C'$, for some constant $C'$ depending only on $A$ and $s=\# S$.  Note that $Y_{1, 0}= I(b_m, C_0)$ is an additive subgroup of $A$, and we can ensure that $Y_{1, 0}$ contains at least $2\cdot q^{-4s(\deg(T_1)+\deg(T_2)+\cdots+\deg(T_m)}$ elements simply by choosing $C_m$ larger than some value which depends only on $A$, $r$, and $s$. 

Now, for each $1\leq i\leq m$, enumerate the places in $S_{1, i}=\{v_1, v_2, ..., v_{k_i}\}$, and set $Y_{i, 0}=Y_{i-1, k_{i-1}}$ (where $Y_{1, 0}$ is as defined above).
We may apply Lemma~\ref{lem:localchoose}, for each $1\leq t\leq k_i$ to obtain an additive subgroup $Y_{i, t}\subseteq Y_{i, t-1}$ such that $\# Y_{t, i}\geq q^{-4r^2\deg(T_i)} \# Y_{i, t-1}$, and such that for all $a\in Y_{i, t}$ we have $\phi_a(x)=0$ or
\[-\log|\phi_a(x)|_{v_t}+c_{v_t}(\phi)\geq (1-q^{-1})j_{\phi_{T_i}, v_t}-\frac{1}{q(q^{r\deg(T_i)}-1)^2}\log^+|T_i^{-1}|.\]
To recap, we have an additive subgroup $Y_{i, k_i}\subseteq A$, with \[\#Y_{i, k_i}\geq q^{-4k_ir^2\deg(T_i)}\# Y_{i, 0}\geq q^{-4s\deg(T_i)}\# Y_{i, 0},\] such that for all $v\in S_{1, i}$ and $a\in Y_{i, k_i}$, either $\phi_a(x)=0$ or else
\begin{gather*}
\lambda_{\phi, v}(\phi_a(x)) \geq (1-q^{-1})j_{\phi_{T_i}, v}-\frac{1}{q(q^{r\deg(T_i)}-1)^2}\log^+|T_i^{-1}|\\
G_{\phi, v}(\phi_a(x))\geq 0.
\end{gather*}
But it is true, by Lemma~\ref{lem:greenslower}, that for all $v\in S_{2, i}$ and $a\in Y_{i, k_i}$, we have $a=0$ or 
\begin{gather*}
\lambda_{\phi, v}(\phi_a(x))=0\\
G_{\phi, v}(\phi_a(x))\geq \left(\frac{q-1}{q^{r\deg(T_i)}-1}\right)j_{\phi_{T_i}, v}.
\end{gather*}
In particular, for all $v\in S$ and all $a\in Y_{m, k_m}$, we get $\phi_a(x)=0$ or 
\[\lambda_{\phi, v}(\phi_a(x))+G_{\phi, v}(\phi_a(x))\geq \left(\frac{q-1}{q^{r\deg(T_i)}-1}\right)j_{\phi_{T_i}, v}-\frac{1}{q(q^{r\deg(T_i)}-1)^2}\log^+|T_i^{-1}|\]
for all $1\leq i\leq m$, and so
\begin{equation}\label{eq:boundwithj}\lambda_{\phi, v}(\phi_a(x))+G_{\phi, v}(\phi_a(x))\geq \left(\frac{q-1}{q^{R}-1}\right)j_{\phi, v}-\delta_v,\end{equation}
where $R=r\max_{1\leq i\leq m}\deg(T_i)$ and where \[\delta_v=\max_{1\leq i\leq m}\frac{1}{q(q^{r\deg(T_i)}-1)^2}\log^+|T_i^{-1}|.\] 
Note that, by Proposition~\ref{prop:disclocal}, we have
\[j_{\phi, v}>\frac{1}{d+1}\max\{j_{\phi, v}, \Disc_{\phi, v}\}\]
when $j_{\phi, v}>0$, and so \eqref{eq:boundwithj} is no weaker than
\[\lambda_{\phi, v}(\phi_a(x))+G_{\phi, v}(\phi_a(x))\geq \left(\frac{q-1}{q^{R}-1}\right)\left(\frac{1}{d+1}\right)\max\{j_{\phi, v}, \Disc_{\phi, v}\}-\delta_v.\]

We may now choose $C_m$ large enough, in a way which depends only on $A$ and $\#S$, so that $\# Y_{m, k_m}\geq 2$.
If we let $a\in Y_{m, k_m}$ be non-zero, then summing over all places gives
\begin{eqnarray*}
|a|_\infty^r\hat{h}_\phi(x)&=&\hat{h}_\phi(\phi_a(x))\\
&=&\sum_{v\in M_L}\frac{1}{2}\left(\lambda_{\phi, v}(\phi_a(x))+G_{\phi, v}(\phi_a(x))\right)\\
&\geq&\sum_{v\in M_L}\frac{1}{2}\left(\left(\frac{q-1}{q^{R}-1}\right)\left(\frac{1}{d+1}\right)\max\{j_{\phi, v}, \Disc_{\phi, v}\}-\delta_v\right)\\
&\geq&\frac{1}{2}\left(\frac{q-1}{q^{R}-1}\right)\left(\frac{1}{d+1}\right)\max\{h(j_\phi), \deg(\Disc_{\phi/L})\}-C,
\end{eqnarray*}
unless $\phi_a(x)=0$, where
\[C\leq \frac{1}{2}\sum_{v\in M_L}\frac{1}{q(q^r-1)^2}\log^+|T_i^{-1}|\leq \frac{[L:K]}{q(q^r-1)^2}\sum_{1\leq i\leq m}\deg(T_i).\]
Since $a\in Y_{m, k_m}\subseteq I(1, C_m)$, we have that $|a|_\infty\leq C_m$, and the latter quantity is chosen depending only on $\#S$ and the ring $A$, and so we have the claimed inequality.  If, on the other hand, we have $\phi_a(x)=0$, where $a\in Y_{m, k_m}\setminus\{0\}$, then $x\in \phi[aA]\subseteq\phi[\mathfrak{a}]$, where $\mathfrak{a}$ is the least common multiple of all ideal of the form $bA$ with $|b|_\infty\leq C_m$.  Since $C_m$ depends only on $\#S$ and the ring $A$, so does the ideal $\mathfrak{a}$.
\end{proof}

Theorem~\ref{th:jplaces} is a corollary of Theorem~\ref{th:generaljplaces}, once the following lemma is observed.
\begin{lemma}\label{lem:lowernorthcott}
Every $L$-isomorphism class of Drinfeld modules contains an integral model $\phi/L$ such that
\begin{equation}\label{eq:twoheights}h(\phi)\leq 2\max\{h(j_\phi), \deg(\Disc_{\phi/L})\} +g(L)+[L:K]-1,\end{equation}
where $g(L)$ is the genus of $L$.
In particular, for any $B$, there are at most finitely many $L$-isomorphism classes of Drinfeld modules $\phi/L$ of rank $r$ satisfying
\[\max\{h(j_\phi), \deg(\Disc_{\phi/L})\}\leq B.\]
\end{lemma}

\begin{proof}
The proof is similar to that of Proposition~\ref{lem:classgp}, but slightly different as the definition of quasi-minimality takes into account the value of $c_v(\phi)$ only at finite places, whereas the definition of $h(\phi)$ involves infinite places, as well.

Suppose that $\phi/L$ is an integral model such that $h(\phi)$ is minimal within the $L$-isomorphism class of $\phi$, and let
\[D=\sum_{v\in M_L}\frac{\max\{0, c_v(\phi)\}}{\deg(v)}[v]-\Disc_{\phi/L}\geq 0.\]
Since $\phi$ is an integral model, $c_v(\phi)\geq 0$ for all $v\in M_L^0$.
Notice, though, that if $c_v(\phi)\geq 0$ for all $v\in M_L^\infty$, then $c_v(\phi)=0$ for all $v\in M_L$, by the product formula, hence $h(\phi)=h(j_\phi)$ and inequality~\eqref{eq:twoheights} is trivial.  There is no loss of generality, then, if we suppose that there is some $w\in M_L^\infty$ such that $c_w(\phi)<0$. 

Now suppose that there exists a $\beta\in L\setminus\{0\}$ satisfying $(\beta)\leq D-[w]$, and let $\psi=\beta\phi\beta^{-1}$.  For every finite place $v$, we have
\[c_v(\psi)=c_v(\phi)-v(\beta)\deg(v)\geq \Disc_{\phi, v},\]
and so $\psi$ is $v$-integral.  We also have, for any $v\in M_L$,
\[\max\{0, c_v(\psi)\}+v(\beta)\deg(v)=\max\{v(\beta)\deg(v), c_v(\phi)\}\leq \max\{0, c_v(\phi)\}.\]
For the place $w$, we have the strict inequality
\[\max\{0, c_w(\psi)\}+w(\beta)\deg(w)=\max\{w(\beta)\deg(w), c_w(\phi)\}<0= \max\{0, c_w(\phi)\},\]
since both $c_w(\phi)$ and $w(\beta)$ are negative.
By the product rule, we have
\begin{eqnarray*}
h(\psi)&=&h(j_\psi)+\sum_{v\in M_L}\max\{0, c_v(\psi)\}\\
&=&h(j_\phi)+\sum_{v\in M_L}\left(\max\{0, c_v(\psi)\}+v(\beta)\deg(v)\right)\\
&<&h(j_\phi)+\sum_{v\in M_L}\max\{0, c_v(\phi)\}=h(\phi),
\end{eqnarray*}
a contradiction.

We have shown, then, that if $\phi/L$ is an integral model such that $h(\phi)$ is minimal with the $L$-isomorphism class, then the Riemann-Roch space $\mathscr{L}(D-[w])$ is trivial, and hence $\deg(D-[w])\leq g(L)-1$.  From this we conclude that
\begin{eqnarray*}
h(\phi)&=&h(j_\phi)+\sum_{v\in M_L}\max\{0, c_v(\phi)\}\\
&=&h(j_\phi)+\deg(\Disc_{\phi/L})+\deg(D)\\
&\leq& 2\max\{h(j_\phi), \deg(\Disc_{\phi/L})\}+g(L)+\deg(w)-1\\
&\leq& 2\max\{h(j_\phi), \deg(\Disc_{\phi/L})\}+g(L)+[L:K]-1.
\end{eqnarray*}

For the second claim, that there are only finitely many $L$-isomorphism classes satisfying
\[\max\{h(j_\phi), \deg(\Disc_{\phi/L})\}\leq B,\]
we may now simply invoke Proposition~\ref{prop:northcott}.

\end{proof}

\begin{proof}[Proof of Theorem~\ref{th:jplaces}]
Under the hypotheses of Theorem~\ref{th:jplaces}, we have by Theorem~\ref{th:generaljplaces} constants $\epsilon$ and $C$ and an ideal $\mathfrak{a}\subseteq A$ such that for every $\phi/L$ of rank $r$ with at most $s$ places of bad reduction, we have
have either $x\in \phi[\mathfrak{a}]$ or else
\[\hat{h}_\phi(x)\geq \epsilon\max\{h(j_\phi), \deg(\Disc_{\phi/L})\}-C,\]
for each $x\in \phi(L)$.
The latter case implies
\[\hat{h}_\phi(x)\geq \frac{\epsilon}{2}\max\{h(j_\phi), \deg(\Disc_{\phi/L})\},\]
unless $\phi/L$ satisfies $\max\{h(j_\phi), \deg(\Disc_{\phi/L})\}\leq 2C/\epsilon$.  By Lemma~\ref{lem:lowernorthcott}, there are only finitely many $L$-isomorphism classes of $\phi/L$ satisfying this inequality.

Now, for each of these finitely many $L$-isomorphism classes, since the annihilator of the torsion submodule of $\phi(L)$ and the minimum positive value of $\hat{h}_\phi$ on $\phi(L)$ are both $L$-isomorphism invariants, we may adjust $\epsilon$ downward and $\mathfrak{a}$ upward (in the sense of divisibility) so that it still true that
\[\hat{h}_\phi(x)\geq \frac{\epsilon}{2}\max\{h(j_\phi), \deg(\Disc_{\phi/L})\},\]
unless $x\in \phi[\mathfrak{a}]$.  This proves Theorem~\ref{th:jplaces}.
\end{proof}

We conclude this section with a proof of Theorem~\ref{th:sharpness}.
\begin{proof}[Proof of Theorem~\ref{th:sharpness}]
Note that both the minimal positive value of $\hat{h}_\phi(x)$, as $x\in\phi(L)$ varies over non-torsion points, and the quantity $\max\{h(j_\phi), \deg(\Disc_{\phi/L})\}$ are $L$-isomorphism invariants.  In particular, there is no loss of generality if we prove Theorem~\ref{th:sharpness} only in the case where $\phi/L$ is such that $h(\phi)\leq h(\psi)$ for any $\psi/L$ which is $L$-isomorphic to $\phi$.

Now, it follows from the assumption of uniform boundedness of torsion that there is a constant $B$, depending only on $L$ and $r$, such that every Drinfeld module $\phi/L$ of rank $r$ has a non-torsion point $x\in\phi(L)$ satisfying $h(x)\leq B$.  Indeed, if we assume that $\#\phi^{\mathrm{Tors}}(L)\leq N$ for every $\phi/L$ of rank $r$, then it suffices to choose $B$ large enough that $L$ contains more than $N$ elements of height at most $B$.  For quasi-minimal $\phi/L$, it follows from this, Theorem~\ref{th:zimmer}, and Lemma~\ref{lem:lowernorthcott} that there exists a non-torsion $x\in\phi(L)$ such that
\[\hat{h}_\phi(x)\leq h(x)+h(\phi)+O(1)\leq h(\phi)+O(1)\leq 2\max\{h(j_\phi), \deg(\Disc_{\phi/L})\} +O(1),\]
where the implied constants depend only on $A$, $r$, and $L$.
\end{proof}


\section{Drinfeld modules in families}
\label{sec:families}

In this section we turn our attention to the proof of Theorem~\ref{th:fams}.  The following fact is a modification of a result in \cite{variation}, and the proof is nearly identical.

\begin{theorem}\label{th:variation}
Let $L$ be a global $A$-field, let $X/\reas{L}$ be a curve, let $\phi/\reas{L}(X)$ be a Drinfeld $A$-module, and let $x\in \phi(\reas{L}(X))$.  Then there exists a divisor $D\in \Pic(X)\otimes\QQ$ such that for all $\beta\in X(\reas{L})$,
\[\hat{h}_{\phi_\beta}(x_\beta)=h_D(\beta)+O(1).\]
Furthermore, $\deg(D)=\hat{h}_\phi(x)=0$ only if $x\in \phi^{\mathrm{Tors}}(\reas{L})$, or $\phi$ is isotrivial.
\end{theorem}

As noted, the proof of this result is essentially contained in \cite{variation}.  Note that the the arguments in \cite{variation} assume the context of a number field, but it is only in proving the local analyticity of local heights that the characteristic of the field enters into the argument.  The fact that $\deg(D)=\hat{h}_\phi(x)$ follows directly from the construction in \cite{variation}, and the fact that this vanishes only if $x\in\phi^{\mathrm{Tors}}(\reas{L})$, or $\phi$ is isotrivial follows from a result of Ghioca \cite{ghioca1} (see also Baker~\cite{baker}).

\begin{proof}[Proof of Theorem~\ref{th:fams}]
By the main result of \cite{variation}, there is an effective divisor $D$ on $C$ of degree $\hat{h}_\phi(x)$ such that
\[\hat{h}_{\phi_\beta}(x_\beta)=h_D(\beta)+O(1),\]
and  $D$ is ample as long as $\phi$ is not isotrivial.  
The set of points at which the left-hand side vanishes, then, is a set of bounded height (with respect to the ample divisor $D$).

On the other hand, let $E$ be the pole divisor of $j_\phi$, so that $\phi_\beta$ has bad reduction only at places such that $\beta$ is not integral with respect to $E$.  If we fix this number of places, we have for any degree one height $h$ on $C$,
\[\epsilon \deg(j_\phi)h(t)=\epsilon h(j_{\phi_t})\leq \hat{h}_{\phi_t}(y)+\kappa,\]
for all $y\in L$, unless $y\in \phi_\beta^{\mathrm{Tors}}(L)$.
So if $y\in \Gamma_{\phi_\beta}(x_\beta, L)$, then either $y\in\phi_\beta^{\mathrm{Tors}}(L)$ or
 $\phi_{\beta, b}(y)=x_\beta$, for some $b\in A$.  Then we have
\[\epsilon \deg(j_\phi)h(t)\leq\hat{h}_{\phi_\beta}(y)+\kappa=|b|_\infty^{-r}\hat{h}_{\phi_\beta}(x_\beta)+O(1)=|b|_\infty^{-r}h_D(\beta)+O(1).\]
Dividing by $h(t)$, we have
\[|b|_\infty^r\leq \frac{\deg(D)}{\epsilon\deg(j_\phi)}+o(1),\]
where $o(1)\to 0$ as $h(t)\to\infty$.  This gives us only finitely many choices for $b$, once we exclude parameters coming from some set of bounded height.

So we've shown that if $y\in \Gamma_{\phi_\beta}(x_\beta, L)$, then either $\phi_{\beta, b}(y)\in \langle x_\beta\rangle$ for some  $b$ of degree bounded in terms of the number of places at which $\beta$ is not integral with respect to $E$, or else $y\in\phi_{\beta}^{\mathrm{Tors}}(L)$.  But the latter is bounded in size by the number of places at which $\beta$ is not integral with respect to $E$, and so the index $(\Gamma_{\phi_\beta}(x_\beta, L):\langle x_\beta \rangle)$ is similarly bounded.
\end{proof}


\end{document}